\documentclass[12pt]{amsart}
\usepackage{a4wide}
\usepackage[utf8]{inputenc}
\usepackage{amsmath}
\usepackage{amssymb}
\usepackage{amsfonts}
\usepackage{amsopn}
\usepackage{graphicx}
\usepackage{enumerate}
\usepackage[dvipsnames]{xcolor}
\usepackage{mathtools}
\usepackage{tikz}
\usepackage{mathrsfs}
\usepackage{bbm}
\usepackage{yhmath}
\usepackage{cite}
\usepackage[colorlinks,linkcolor=blue,anchorcolor=blue,citecolor=blue]{hyperref}
\usepackage{amsmath,amsthm,amssymb,stmaryrd,latexsym,a4wide,epic,eepic,bbm,mathrsfs,amsfonts,cases,yfonts,indentfirst,bbold}
\usepackage{dsfont,color,mathdots}

\usepackage{amsmath,amssymb}
\usepackage{tablists}
   \restorelistitem
\usepackage{amsfonts}
\usepackage{tipa}
\usepackage{CJK, CJKnumb}
\usepackage{color}     
\usepackage{indentfirst}      
\usepackage{latexsym, bm}     
\usepackage{graphicx}
\usepackage{cases}
\usepackage{fancyhdr}
\usepackage{pifont}
\usepackage{lineno,hyperref}
\modulolinenumbers[5]
\usepackage{amsmath}

\newtheorem{theorem}{Theorem}[section]
\newtheorem{defi}[theorem]{Definition}
\newtheorem{lemma}[theorem]{Lemma}
\newtheorem{coro}[theorem]{Corollary}
\newtheorem{proposition}[theorem]{Proposition}

\newtheorem{remark}[theorem]{Remark}
\usepackage[active]{srcltx}
\usepackage{enumerate}
\numberwithin{equation}{section}

\usepackage[T1]{fontenc}
\usepackage{lmodern}
\usepackage{amsmath,amsthm,amssymb,stmaryrd,latexsym,a4wide,epic,eepic,bbm,mathrsfs,amsfonts,cases,yfonts,indentfirst,bbold}
\usepackage{dsfont,color,mathdots}
\numberwithin{equation}{section}

\title{Drinfeld Isomorphism for Novel Quantum Affine Algebra of Type \(A_{1}^{(1)}\)}

\author[R.S. Zhuang]{Rushu Zhuang}
\address[R.S.Zhuang]{School of Mathematics and Physics, China University of Geosciences, Wuhan 430074, China}
\email{zhuangrushu@cug.edu.cn}

\author[G. Feng]{Ge Feng}
\address[G. Feng]{College of Science, University of Shanghai for Science and Technology, Shanghai 200093, China}
\email{fengge@usst.edu.cn}

\author[N.H. Hu]{Naihong Hu$^*$}
\address[N.H. Hu]{School of Mathematical Sciences, MOE Key Laboratory of Mathematics and Engineering Applications $\&$ Shanghai Key Laboratory of PMMP, East China Normal University, Shanghai 200241, China}
\email{nhhu@math.ecnu.edu.cn}

\thanks{*Corresponding author}

\begin{document}

\begin{abstract}
In this paper, we first review the definition of the novel quantum affine algebra \(U_{\textbf{q}}(\widehat{\mathfrak{sl}}_2)\) of type \(A_{1}^{(1)}\) given in \cite{FHZ, HZhuang}. Furthermore, by introducing \(\Omega\)-invariant generating functions, we construct the Drinfeld realization \(U^{D}_{\textbf{q}}(\widehat{\mathfrak{sl}}_2)\) of this algebra, and prove that \(U_{\textbf{q}}(\widehat{\mathfrak{sl}}_2)\) and \(U^{D}_{\textbf{q}}(\widehat{\mathfrak{sl}}_2)\) are algebraically isomorphic, which is known as the Drinfeld Isomorphism.
\end{abstract}


\maketitle
\tableofcontents

\section{Introduction}
In 2000, Benkart and Witherspoon revitalized the research of two-parameter quantum groups. They studied the structures of two-parameter quantum groups \(U_{r, s}(\mathfrak{g})\) for \(\mathfrak{g}=\mathfrak{gl}_n\) or \(\mathfrak{sl}_n\) in \cite{BW1} previously obtained by Takeuchi \cite{T}, and the finite-dimensional representations and Schur-Weyl duality for type A in \cite{BW2}.
Later, Bergeron, Gao and Hu \cite{BGH} gave a presentation of two-parameter quantum groups \(U_{r,s}(\mathfrak{g})\) for the classical types \(\mathfrak{so}_{2n+1}\), \(\mathfrak{sp}_{2n}\) and \(\mathfrak{so}_{2n}\), and 
further investigated the environment conditions upon which the Lusztig’s symmetries existed between
(\(U_{r,s}(\mathfrak{g}), \langle,\rangle\)) and its associated object (\(U_{s^{-1},r^{-1}}(\mathfrak{g}), \langle|\rangle\)). 
In \cite{PHR}, Pei-Hu-Rosso studied the multi-parameter quantum groups, together with their representations in category \(\mathcal{O}\). In particular, they presented two explicit descriptions: as Hopf 2-cocycle deformation, and as the multi-parameter quantum shuffle realization of the positive part.
A generalisation to multi-parameter quantum groups was obtained by Heckenberger \cite{H1}, which provided an explicit realization model for the abstract concept "Weyl groupoid" playing a key role in the classification of both Nichols algebras of diagonal type (cf. \cite{H2}) and finite-dimensional pointed Hopf algebras with abelian group algebras as the coradicals (\cite{AS}).

On the other hand, Hu-Rosso-Zhang first studied two-parameter quantum affine algebras associated to affine Lie algebras of type \(A_n^{(1)}\), and gave the descriptions of the structure and Drinfeld realization of \(U_{r, s}(\widehat{\mathfrak{sl}}_n)\), as well as the quantum affine Lyndon basis (see \cite{8}). The discussions for the other affine cases of untwisted types and the corresponding vertex operators constructions for all untwisted types have been done in \cite{HZ,Z}. In \cite{JZ1}, Jing-Zhang used a combinatorial model of Young diagrams, gaving a fermionic realization of the two-parameter quantum affine algebra of type \(A_n^{(1)}\); while \cite{JZ2} provided a group-theoretic realization of two-parameter quantum toroidal algebras using finite subgroups of \(SL_2(\mathbb{C})\) via McKay correspondence.

For a symmetrizable generalized Cartan matrix \(A=(a_{ij})_{1\le i,j\le n}\) (i.e. \(DA\) is symmetric for \(D=\operatorname{diag}(d_1,\dots ,d_n)\) with coprime positive integers \(d_i\)) one has the quantized enveloping algebra \(U_{\mathbf q}(\mathfrak g_A)\) of the Kac-Moody algebra \(\mathfrak g_A\) \cite{PHR}. Here the matrix \(\mathbf q=(q_{ij})\) satisfies the Cartan condition \(q_{ij}q_{ji}=q_{ii}^{a_{ij}}\) and is called the structure constant matrix (or braiding matrix) of \(U_{\mathbf q}(\mathfrak g_A)\).  
In \cite{HZhuang},  from a multi-parameter perspective, that is, their structure constant matrix taking the form \[ \mathbf q=\begin{pmatrix} \epsilon_1 q^{2} & t q^{-2}\\[2pt] t^{-1}q^{-2} & \epsilon_2 q^{2} \end{pmatrix},\qquad  \epsilon_i\in\{\pm1\},\; t\in\mathbb Q(q)^*, \] and satisfying the Cartan condition, we defined a class of sign-deformed quantum affine algebra \(U^{(\varepsilon_1,\varepsilon_2,t)}_q(\widehat{\mathfrak{sl}}_2)\). In particular, we studied its structure in the case when \(\epsilon_1=\epsilon_2=-1\), \(t=1\); while in \cite{FHZ}, we considered the  case \(\epsilon_1=\epsilon_2=1\), \(t=-1\). In both cases, we explicitly described their PBW bases, and proved that they are not isomorphic to the standard affine quantum group of type \(A_1^{(1)}\), as (Hopf) algebras.

The present paper is devoted to working out the Drinfeld realization and demonstrating the Drinfeld Isomorphism for the above quantum affine algebra \(U_{\mathbf q}(\widehat{\mathfrak{sl}}_2)\) in the specific case when \(\epsilon_1=\epsilon_2=-1\), \(t=1\).

In Section~2, we recall the definition of \(U_{\mathbf q}(\widehat{\mathfrak{sl}}_2)\), its Hopf algebra structure, the automorphisms induced by the braid group action, and the definitions of the real and imaginary root vectors and their commutation relations.  
Section~3 is the core of the current work. We construct \(\Omega\)-invariant generating functions and use them to give the intrinsic Drinfeld realisation \(U_{\mathbf q}^{D}(\widehat{\mathfrak{sl}}_2)\). Then we prove that \(U_{\mathbf q}(\widehat{\mathfrak{sl}}_2)\) and \(U_{\mathbf q}^{D}(\widehat{\mathfrak{sl}}_2)\) are isomorphic by exhibiting a pair of mutually inverse algebra homomorphisms (the Drinfeld Isomorphism).

\section{Notations and preliminaries}
In this section, we recall the definition of the novel 
$A^{(1)}_1$-type quantum affine algebra 
\(U_{\textbf{q}}(\widehat{\mathfrak{sl}}_2)\)
and summarize the related results from previous works. For comprehensive details and proofs, we refer the reader to references \cite{FHZ} \& \cite{HZhuang}.

\subsection{Hopf algebra and Drinfeld double.}
Let 
\begin{equation}
A=(a_{ij})_{i,j\in\{1,2\}}=\left(\begin{array}{rr}2&-2\\-2&2\\\end{array}\right)
\end{equation}
denote the generalized Cartan matrix of type \(A^{(1)}_1\). We first present the definition of the novel quantum affine algebra \(U_{\textbf{q}}(\widehat{\mathfrak{sl}}_2)\) in terms of its generators and relations.
\begin{defi} [\cite{HZhuang}, Def. 3.1]
The novel \(A^{(1)}_1\)-type quantum affine algebra \(U_{\textbf{q}}(\widehat{\mathfrak{sl}}_2)\) is the unital associative algebra over the field 
\(\mathbb{C}(\rm q)\) generated by the elements 
\begin{equation}
e_i, f_{i},k^{\pm1}_{i},\quad i=1,2,
\end{equation}
subject to the following defining relations (A1)–(A4):

(A1) The group-like elements \(k_{1}\)
and \(k_{2}\) commute with each other. 

(A2)  Commutation relations between \(k_{i}\)
and \(e_j, f_{j}\):
\begin{align*}
&k_ie_i=-q^2e_ik_i, \quad k_if_i=-q^{-2}f_ik_i,\\
&k_ie_j=q^{-2}e_jk_i, \quad \,k_if_j=q^2f_jk_i, \qquad (i\ne j).
\end{align*}

(A3) Cross relations between \(e_i\) and 
\(f_j\,(i,\;j\in\{1,\;2\})\):
\[[\,e_i,f_j\,]=\delta_{ij}\frac{k_i-k^{-1}_i}{q-q^{-1}}.\]

(A4) \(q\)-Serre relations: for \(i\neq j\),  
\begin{align*}
&(ad_l e_i)^{1-a_{ij}}(e_j)=e_i^3e_j+(1-q^2-q^{-2})e_i^2e_je_i+(1-q^2-q^{-2})e_ie_je_i^2+e_je_i^3=0,\\
&(ad_r f_i)^{1-a_{ij}}(f_j)=f_i^3f_j+(1-q^2-q^{-2})f_i^2f_jf_i+(1-q^2-q^{-2})f_if_jf_i^2+f_jf_i^3=0,
\end{align*}
where the definitions of the left-adjoint action \(ad_l e_i\) and the right-adjoint action
\(ad_r f_i\) are given in the following sense:
\[ad_l a(b)=\sum\limits_{(a)}a_{(1)}bS(a_{(2)}),\qquad ad_r a(b)=\sum\limits_{(a)}S(a_{(1)})ba_{(2)},\quad \forall\, a,b\in U_{\bm q}(\widehat{\mathfrak{sl}}_2).\]
where \(\Delta(a)=\sum\limits_{(a)}a_{(1)}\otimes a_{(2)}\) is given by Proposition 2.2 below.
\end{defi}

\begin{proposition}[\cite{HZhuang}, Def. 3.1]
The algebra \(U_{\textbf{q}}(\widehat{\mathfrak{sl}}_2)\) is a Hopf algebra with the comultiplication $\Delta$, the counit $\varepsilon$, the antipode $S$, satisfying the following relations:
\begin{tabenum}[(B1)]
\qquad\qquad\tabenumitem $\Delta(e_{i})= e_{i}\otimes 1+k_{i}\otimes e_{i},$\qquad
\tabenumitem 
$\Delta(f_{i})= 1\otimes f_i+f_i\otimes k^{-1}_{i},$\\ \tabenumitem 
$\Delta(k_i^{\pm1})= k_i^{\pm1}\otimes k_i^{\pm1},$
\tabenumitem $\varepsilon(e_{i})=\varepsilon(f_{i})=0,$\\
\tabenumitem $\varepsilon(k_i^{\pm1})=1,$
\tabenumitem $S(k_i^{\pm 1})=k_i^{\mp 1},$\\
\tabenumitem $S(e_{i})=-k_i^{-1}e_{i},$\quad
\tabenumitem $S(f_{i})=-f_{i}k_i.$
\end{tabenum}
\end{proposition}
Let \(\widehat{\mathcal{B}}\,\) 
 (resp. \(\widehat{\mathcal{B}^{\prime}}\)) denote the Hopf subalgebra of \(\,U_{\textbf{q}}(\widehat{\mathfrak{sl}}_2)\) generated by \(e_i,k^{\pm1}_i\,\) (resp. \(f_i, k^{\pm1}_i\)) with \(1\le i\le2\). 

\begin{proposition}[\cite{HZhuang}, Prop. 2.3]
There exists a unique non-degenerate skew-dual pairing \(\langle , \rangle: \widehat{\mathcal{B}^{\prime}} \times \widehat{\mathcal{B}} \to \mathbb{Q}(q)\) of the Hopf subalgebras \(\widehat{\mathcal{B}}\) and \(\widehat{\mathcal{B}^{\prime}}\) such that
\begin{itemize}
\item 
\(\langle f_i, e_j \rangle = \dfrac{\delta_{ij}}{q^{-1} - q} \quad (i, j \in \{1, 2\})\);
\item \(\langle k^{\pm1}_i, k_j^{\mp 1} \rangle = {\langle k^{\pm1}_i, k_j^{\pm 1} \rangle}^{-1}=q_{ij}\), for all \(i, j \in \{1, 2\}\);
\item All other pairs of generators take the value \(0\).
\end{itemize}
Moreover, we have
\[
\langle S(a), S(b) \rangle = \langle a, b \rangle
\]
for \(a \in \widehat{\mathcal{B}^{\prime}}\) and \(b \in \widehat{\mathcal{B}}\), where \(S\) denotes the antipode of \(U_{\textbf{q}}(\widehat{\mathfrak{sl}}_2)\).
\end{proposition}

\begin{proposition}[\cite{HZhuang}, Prop. 2.4]
The quantum affine algebra \(U_{\textbf{q}}(\widehat{\mathfrak{sl}}_2)\) is isomorphic to the Drinfeld quantum double \(\mathcal{D(\hat{B},\hat{B}^{\prime})}\).
\end{proposition}
\subsection{Automorphisms of 
\(U_{\textbf{q}}(\widehat{\mathfrak{sl}}_2)\)}
The novel quantum affine algebra \(U_{\textbf{q}}(\widehat{\mathfrak{sl}}_2)\)
possesses automorphisms, which play a crucial role in deriving root vectors and understanding the structure of the algebra. In particular, we focus on automorphisms induced by the braid group action.

\begin{proposition}[\cite{HZhuang}, Def. 3.2]
For \(i,\;j\in\{1,\;2\}\), the mappings \(\mathcal{T}_i\)
defined by the following formulas are automorphisms of the algebra \(U_{\textbf{q}}(\widehat{\mathfrak{sl}}_2)\):
\begin{flalign*}
&\mathcal{T}_{i}(k_{i})=k^{-1}_{i},\qquad \mathcal{T}_{i}(k_{j})=k_{j}k^{-a_{ij}}_{i},\\
&\mathcal{T}_{i}(e_{i})=-f_{i}k_{i},\quad \mathcal{T}_{i}(f_{i})=-k^{-1}_{i}e_{i},\\
&\mathcal{T}_{i}(e_{j})=e^{(2)}_{i}e_{j}-q^{-2}e_{j}e^{(2)}_{i}+\frac{1-q^{-2}}{[2]_{q}}e_{i}e_{j}e_{i},\\
&\mathcal{T}_{i}(f_{j})=-q^{2}f^{(2)}_{i}f_{j}+f_{j}f^{(2)}_{i}+\frac{1-q^{2}}{[2]_{q}}f_{i}f_{j}f_{i},\qquad i\neq j.
\end{flalign*}
Here, we use the standard q-deformed notation:
\begin{flalign*}
e^{(n)}_{i}=\frac{e^{n}_{i}}{[n]_{q}!},\qquad [n]_{q}=\frac{q^{n}-q^{-n}}{q-q^{-1}},\qquad [n]_{q}!=[1]_{q}\cdots[n]_{q},\quad n\in\mathbb{N}.
\end{flalign*}
\end{proposition}

\begin{proposition}
There exists a (Hopf) \(\mathbb{C}[q]\)-automorphism \(\Phi\) of \(U_{\textbf{q}}(\widehat{\mathfrak{sl}}_2)\), which interchanges the indices 1 and 2:
\[\Phi(e_i)=e_j,\quad \Phi(f_i)=f_j,\quad \Phi(k_i)=k_j,\quad i\neq j,\]
and satisfy \(\Phi^2=id,\,\mathcal{T}_1\Phi=\Phi \mathcal{T}_2.\)
\end{proposition}
\begin{proposition}
There exists a \(\mathbb{C}\)-antiautomorphism \(\Omega\) of \(U_{\textbf{q}}(\widehat{\mathfrak{sl}}_2)\) defined by 
\[\Omega(e_i)=f_i,\quad \Omega(f_i)=e_i,\quad \Omega(k_i)=k^{-1}_i,\quad \Omega(q)=q^{-1},\quad i=1,2.\]
Moreover, \(\Omega^2=id,\,\Omega \mathcal{T}_i=\mathcal{T}_i\Omega\) and \(\Phi\Omega=\Omega\Phi.\)
\end{proposition}
\subsection{Definition of Root Vectors and Commutation Rules}
Root vectors are essential for understanding the structure of quantum affine algebras, as they correspond to the roots of the underlying affine Kac-Moody algebra and generate the root spaces of the quantum affine algebra. The root system of 
\(A^{(1)}_1\) consists of real roots (of the form \(n\delta+\alpha_i\) for \(n\in\mathbb{Z}_{\ge 0}\) and \(i=1,2\), where 
\(\alpha_1,\alpha_2\) are the simple roots and \(\delta\) is the null root) and imaginary roots (of the form \(n\delta\) for 
\(n\in\mathbb{Z}\setminus\{0\}\)).

We first define the real root vectors using the automorphisms \(\mathcal{T}_{i}\) introduced in Proposition 2.4, which generalize the Weyl group action on the root system:
\begin{align*}
&\mathcal{E}_{n\delta+\alpha_{1}}=\underbrace{\mathcal{T}_{1}\mathcal{T}_{2}\cdots \mathcal{T}_{i}}_{n}(e_{j})=(\mathcal{T}_1\Phi)^n(e_1),\quad(i\neq j),\\
&\mathcal{E}_{n\delta+\alpha_{2}}=\underbrace{\mathcal{T}^{-1}_{2}\mathcal{T}^{-1}_{1}\cdots \mathcal{T}^{-1}_{i}}_{n}(e_{j})=(\mathcal{T}^{-1}_2\Phi)^n(e_2),\quad(i\neq j),\\
&\mathcal{F}_{n\delta+\alpha_{1}}=\Omega(\mathcal{E}_{n\delta+\alpha_{1}}),\qquad\mathcal{F}_{n\delta+\alpha_{2}}=\Omega(\mathcal{E}_{n\delta+\alpha_{2}}).
\end{align*}
Next, we define the imaginary root vectors \(\mathcal{E}_{n\delta}\,(n \in \mathbb{N})\) using the adjoint action of \(e_1\) on the real root vectors, as proposed in \cite{HZhuang}:
\[\mathcal{E}_{n\delta}={ad}_l e_{1}(\mathcal{E}_{(n-1)\delta+\alpha_{2}})=e_{1}\mathcal{E}_{(n-1)\delta+\alpha_{2}}+{(-1)}^{n}q^{-2}\mathcal{E}_{(n-1)\delta+\alpha_{2}}e_{1}.\]

Accordingly, we have
\[\mathcal{F}_{n\delta}={ad}_r f_{1}(\mathcal{F}_{(n-1)\delta+\alpha_{2}})=\mathcal{F}_{(n-1)\delta+\alpha_{2}}f_{1}+{(-1)}^{n}q^{2}f_{1}\mathcal{F}_{(n-1)\delta+\alpha_{2}}.\]
\begin{theorem} [\cite{HZhuang}, Thm. 4.1]
\(\mathcal{T}_1\mathcal{T}_2(\mathcal{E}_{\delta})=\mathcal{E}_{\delta}\) and for any \(n\in\mathbb{Z}_{\ge 0}\), we have
\begin{align*}
&\mathcal{E}_{\delta}\underbrace{\mathcal{T}_{1}\mathcal{T}_{2}\cdots \mathcal{T}_{i}}_{n}(\mathcal{E}_{j})+\underbrace{\mathcal{T}_{1}T_{2}\cdots \mathcal{T}_{i}}_{n}(\mathcal{E}_{j})\mathcal{E}_{\delta}=[2]_q\underbrace{\mathcal{T}_{1}\mathcal{T}_{2}\cdots \mathcal{T}_{i}}_{n+1}(\mathcal{E}_{j}),\\
&\mathcal{E}_{\delta}\underbrace{\mathcal{T}^{-1}_{2}\mathcal{T}^{-1}_{1}\cdots \mathcal{T}^{-1}_{i}}_{n}(\mathcal{E}_{j})+\underbrace{\mathcal{T}^{-1}_{2}\mathcal{T}^{-1}_{1}\cdots \mathcal{T}^{-1}_{i}}_{n}(\mathcal{E}_{j})\mathcal{E}_{\delta}=[2]_q\underbrace{\mathcal{T}^{-1}_{2}\mathcal{T}^{-1}_{1}\cdots \mathcal{T}^{-1}_{i}}_{n+1}(\mathcal{E}_{j}),
\end{align*}
where $i\neq j$.
\end{theorem}
\begin{coro} [\cite{HZhuang}, Coro. 4.10]
For any $n,m\in\mathbb{Z}_{\ge 0}$, we have
$$
\mathcal{E}_{(n+m+1)\delta}=\mathcal{E}_{n\delta+\alpha_{1}}\mathcal{E}_{m\delta+\alpha_{2}}+{(-1)}^{n+m+1}q^{-2}\mathcal{E}_{m\delta+\alpha_{2}}\mathcal{E}_{n\delta+\alpha_{1}}.
$$
\end{coro}
\begin{proposition}[\cite{HZhuang}, Prop. 4.3]
The following statements are equivalent

$(1)\hspace{2mm}[\mathcal{E}_{\delta}, \mathcal{E}_{n\delta}]=0,\,\forall \,n\in\mathbb{Z}_+$.

$(2)\hspace{2mm}\mathcal{E}_{(n+m+1)\delta}=\mathcal{E}_{n\delta+\alpha_{1}}\mathcal{E}_{m\delta+\alpha_{2}}+{(-1)}^{n+m+1}q^{-2}\mathcal{E}_{m\delta+\alpha_{2}}\mathcal{E}_{n\delta+\alpha_{1}},\,\forall \,n,m\in\mathbb{Z}_+$.

$(3)\hspace{2mm} \mathcal{T}_1\mathcal{T}_2(\mathcal{E}_{n\delta})=\mathcal{E}_{n\delta},\,\forall \,n\in\mathbb{Z}_+$.
\end{proposition}

\section{Drinfeld realization of \(U_{\textbf{q}}(\widehat{\mathfrak{sl}}_2)\)}
In order to obtain the intrinsic definition of Drinfled realization of \(  U_{\textbf{q}}(\widehat{\mathfrak{sl}}_2)\), we need to construct the generating function \(g^{\pm}(z)\) with \(\Omega\)-invariance, which is defined as follows.

We set \(g^{\pm}(z)=\sum\limits_{n\in\mathbb{Z}_+}c^{\pm}_n z^n\) a formal power series in \(z\), which can be expressed as follows \cite{HRZ, HZ}
\[g^{\pm}(z)=\frac{q^{\pm 2}z-1}{z-q^{\pm 2}}.\]

Before presenting the definition of the Drinfeld realization of \(\,U_{\textbf{q}}(\widehat{\mathfrak{sl}}_2)\), 
let us introduce the other generating functions in the variable \(z\) as follows (also see \cite{HRZ,HZ})
\begin{align*} &\delta(z)=\sum\limits_{n\in\mathbb{Z}}z^n,\hspace{2.3cm} x^{\pm}(z)=\sum\limits_{k\in\mathbb{Z}}x^{\pm}(k)z^{-k},\\ 
&k^+(z)=\sum\limits_{m=0}^{+\infty}k^+(m) z^{-m}=k_2 \exp \Big( (q{-}q^{-1})\sum\limits_{\ell=1}^{+\infty}  a(\ell)z^{-\ell}\Big),\\  
&k^-(z)=\sum\limits_{m=0}^{+\infty}k^-(-m) z^{m}=\exp \Big({-}(q{-}q^{-1}) \sum\limits_{\ell=1}^{+\infty}a(-\ell)z^{\ell}\Big)k^{-1}_2.
\end{align*}
\begin{proposition} 
Assume \(\Omega(q^{\pm1})=q^{\mp1}, \,\Omega(z)=z^{-1}, \Omega(x^{\pm}(k))=x^{\mp}(-k), \,\Omega(k^{\pm}(m))=k^{\mp}(\mp m)\), for \(m\in\mathbb{Z}_+\) with \(k^+(0)=k_2,\, k^-(0)=k^{-1}_2,\) then\\
\((\text{\rm i})\quad\Omega(g^{\pm}(z))=g^{\pm}(z), \,\text{and} \,\,g^{\pm}(z)^{-1}=g^{\mp}(z)=g^{\pm}(z^{-1}).\)\\
\((\text{\rm ii})\quad \Omega(\delta(z))=\delta(z),\, \Omega(x^{\pm}(z))=x^{\mp}(z),\, \Omega(k^{\pm}(z))=k^{\mp}(z).\)
\end{proposition}
Now let us fomulate the inherent definition of Drinfeld realization for quantum affine algebra \(U_{\textbf{q}}(\widehat{\mathfrak{sl}}_2)\) via our generating functions with \(\Omega\)-invariance.
\begin{defi}  
The Drinfeld realization \(U^{D}_{\textbf{q}}(\widehat{\mathfrak{sl}}_2)\) associated to the quantum affine algebra
\(U_{\textbf{q}}(\widehat{\mathfrak{sl}}_2)\)
is the associative algebra with unit
\(1\) and generators  \(\bigl\{ x^{\pm}(k), \, k^{\pm}(m), \, \gamma^{\pm 1}, k^{\pm 1}_2 \big|\,~k \in
\mathbb{Z}, \, m \in \mathbb{Z}_+\bigr\}\)
satisfying the relations as below

\((C1)\) \(\gamma^{\pm 1},\,k^{\pm1}_2\) mutually commute, where \(\gamma^{\pm 2n}\) are central and \(k^{\pm1}(0)=k^{\pm1}_2\).

\((C2)\) The following relations involving \(k_2\) hold
\begin{align*}
&k_2k^{\pm}(w) = k^{\pm}(-w)k_2, \qquad
k_2x^{\pm}(w) = -q^{\pm 2}x^{\pm}(-w)k_2.
\end{align*}

\((C3)\) \(k^{\pm}(z)\) and \(k^{\pm}(w)\) satisfy the relation
\[
k^{\pm}(z)k^{\pm}(w) = k^{\pm}(-w)k^{\pm}(-z).
\]

\((C4)\) For \(k^+(z)\) and \(k^-(w)\), we have
\[
\dfrac{(z - q^2\gamma^{-1}w)(z - q^{-2}\gamma w)}{(z - q^{-2}\gamma^{-1}w)(z - q^{2}\gamma w)}k^+(z)k^{-}(w) 
= k^-(-w)k^+(-z).
\]

\((C5)\) The relation between \(k^+(z)\) and \(x^{+}(w)\) is
\[
(z - q^2\gamma^{-1}w)k^+(z)x^{+}(w) 
= (\gamma^{-1}w - q^2z)x^{+}(-w)k^+(z).
\]

\((C6)\) For \(k^+(z)\) and \(x^{-}(w)\), we have
\[
\dfrac{q^2z - w}{q^2w - z}k^+(z)x^{-}(w) 
= x^{-}(-w)k^+(-z).
\]

\((C7)\) The relation between \(x^{+}(z)\) and \(x^{-}(w)\) is given by
\[
\bigl[ x^{+}(z), x^{-}(w) \bigr] 
= \frac{1}{q - q^{-1}}\Bigl( \delta\Bigl(\gamma \frac{w}{z}\Bigr)k^+(w) 
- \delta\Bigl(\frac{w}{z}\gamma^{-1}\Bigr)k^{-}(z) \Bigr).
\]

\((C8)\) \(x^{\pm}(z)\) and \(x^{\pm}(w)\) satisfy the relation
\[x^{\pm}(z)x^{\pm}(w) = g^\pm\Bigl(\frac{z}{w}\Bigr)x^{\pm}(w)x^{\pm}(z).\]
\end{defi}

\begin{defi}\(\bigl({\bf Equivalent\,Definition.}\bigr)\) The unital associative algebra \(U^{D}_{\textbf{q}}(\widehat{\mathfrak{sl}}_2)\)  over \(\mathbb{K}\)  is generated by the elements \(x^{\pm}(k)\), \(a(l)\), \(k_2^{\pm 1}\), \(\gamma^{\pm 1}\)\,\((k \in \mathbb{Z}\), \(l\in \mathbb{Z}\backslash\{0\})\),  subject to the following defining relations

\((D1)\) 
\(\gamma^{\pm 2n}\) are central for \(n \in \mathbb{Z}_{+}\), \(k^{\pm1}_2,\,\gamma^{\pm1}\) mutually commute.

\((D2)\) 
For \(l \in \mathbb{Z}\setminus\{0\}\), \(k \in \mathbb{Z}\),  we have
\begin{align*}  
&\gamma a(l) = a(l)\gamma, \hspace{1.5cm} \gamma x^{\pm}(k)=-x^{\pm}(k)\gamma,\\ &k_2a(l) = (-1)^{l}a(l)k_2,\quad k_2 x^{\pm}(k)  = (-1)^{k+1} q^{\pm 2} x^{\pm}(k) k_2. \end{align*}

\((D3)\)
For \(l, l' \in \mathbb{Z}\setminus\{0\}\), we have
\[
\left[a(l),a(l')\right]=\delta_{l+l',0}\frac{[2l]_q}{|l|}\cdot\frac{\gamma^{|l|}-\gamma^{-|l|}}{q-q^{-1}}.
\]  

\((D4)\) 
For \(l \in \mathbb{Z}_{+}\), \(k \in \mathbb{Z}\), we have
\[
\left[a(l),x^{\pm}(k)\right]=-\dfrac{[2l]_q}{l}x^{\pm}(l+k)\gamma^l.
\]

\((D5)\) 
For \(k, k' \in \mathbb{Z}\), we have
\begin{align*} 
x^{\pm}(k+1)&x^{\pm}(k') - q^2 x^{\pm}(k)x^{\pm}(k' + 1) \\ 
&= (-1)^{k + k' + 1} \Bigl\{ x^{\pm}(k' + 1)x^{\pm}(k)  - q^2 x^{\pm}(k')x^{\pm}(k + 1) \Bigr\}. 
\end{align*}

\((D6)\)
For \(l, l' \in \mathbb{Z}\), we have 
\[
\left[x^+(l),x^-(l')\right]=\frac{1}{q-q^{-1}}\left(\gamma^l k^+(l+l')-\gamma^{-l}k^-(l+l')\right).
\]
\end{defi}

\subsection{Drinfeld Isomorphism} 
The core goal of this section is to establish an algebra isomorphism between the novel quantum affine algebra \(U_{\textbf{q}}(\widehat{\mathfrak{sl}}_{2})\) (defined in Section 2) and its Drinfeld realization \(U^{D}_{\textbf{q}}(\widehat{\mathfrak{sl}}_{2})\) (defined in Section 3). To achieve this, we first state three key theorems, where Theorem 3.5 (surjective homomorphism) and Theorem 3.6 (inverse mapping)  imply  Isomorphism Theorem 3.4. We present the detailed proof of Theorem 3.5 below, and omit that of Theorem 3.6.

\begin{theorem}  
Let \(U_{\textbf{q}}^{D}(\widehat{\mathfrak{sl}}_{2})\) denote the Drinfeld realization of \(U_{\textbf{q}}(\widehat{\mathfrak{sl}}_{2})\). Then we have an algebraic isomorphism:
\[U_{\textbf{q}}(\widehat{\mathfrak{sl}}_{2}) \cong U_{\textbf{q}}^{D}(\widehat{\mathfrak{sl}}_{2}).\]  
\end{theorem}
\begin{theorem} 
There exists a surjective algebra homomorphism \(\Psi: U_{\textbf{q}}^{D}(\widehat{\mathfrak{sl}}_{2}) \to U_{\textbf{q}}(\widehat{\mathfrak{sl}}_{2})\) defined by  
\[
\begin{cases}
\gamma \longmapsto k_{1}k_{2}, \quad k_{2} \longmapsto k_{2}, \\
x^{+}(n) \longmapsto (\mathcal{T}_{1}\Phi)^{n}(e_{1}), \quad x^{-}(n) \longmapsto (\mathcal{T}_{1}\Phi)^{-n}(f_{1}), \\
a(k) \longmapsto (k_{1}k_{2})^{k}\mathcal{E}_{k\delta} - \dfrac{q-q^{-1}}{k}\sum\limits_{r=1}^{k-1}r(k_{1}k_{2})^{k-r}\mathcal{E}_{(k-r)\delta}\Psi(a(r)), \\
a(-k) \longmapsto \mathcal{F}_{k\delta}(k_{1}k_{2})^{-k} + \dfrac{q-q^{-1}}{k}\sum\limits_{r=1}^{k-1}r\Psi(a(-r))\mathcal{F}_{(k-r)\delta}(k_{1}k_{2})^{-k+r},
\end{cases}
\]  
for all \(n \in \mathbb{Z}\) and \(k \in \mathbb{Z}_{+}\). \end{theorem}

\begin{theorem}
There exists a surjective \(\Xi: U_{\textbf{q}}(\widehat{\mathfrak{sl}}_{2}) \to U_{\textbf{q}}^{D}(\widehat{\mathfrak{sl}}_{2})\) such that \(\Psi\Xi=\Xi\Psi=\text{\rm id}\). 
\end{theorem}

\subsection{Proof of Theorem 3.5}
To prove Theorem 3.5, we only need to verify that the mapping \(\Psi\) is a surjective algebra homomorphism. First, the surjectivity of \(\Psi\)
is obvious, since the image of \(\Psi\) exactly constitutes the generator set of \(U_{\textbf{q}}(\widehat{\mathfrak{sl}}_{2})\). Next, we only need to further verify that \(\Psi\)
is an algebra homomorphism.

We rely on commutation relations of root vectors  to check that \(\Psi\)
is an algebra homomorphism.
Before starting the proof, we first state several identities (all obtainable via induction).
\begin{align*}
&(\mathcal{T}_{1}\Phi)^{n}(e_{1})=\mathcal{E}_{n\delta+\alpha_1},\quad (\mathcal{T}_{1}\Phi)^{-n}(e_{1})=-k^{-n+1}_1k^{-n}_2\mathcal{F}_{(n-1)\delta+\alpha_2},\\
&(\mathcal{T}_{1}\Phi)^{n}(f_{1})=\mathcal{F}_{n\delta+\alpha_1},\quad (\mathcal{T}_{1}\Phi)^{-n}(f_{1})=-\mathcal{E}_{(n-1)\delta+\alpha_2}k^{n-1}_1k^n_2, \quad \forall\, n\in\mathbb{Z}_{+}.
\end{align*}

\begin{lemma}
For \(n,m\in\mathbb{Z}_{\ge 0}\), \(n\ge m\), the commutation relations between \(\mathcal{E}_{n\delta+\alpha_{1}}\) and \(\mathcal{E}_{m\delta+\alpha_{1}}\) are 
\begin{align*}     
\mathcal{E}_{n\delta+\alpha_{1}}&\mathcal{E}_{m\delta+\alpha_{1}}+(-1)^{n+m}q^{-2}\mathcal{E}_{m\delta+\alpha_{1}}\mathcal{E}_{n\delta+\alpha_{1}}\\ &=q^{-2}\mathcal{E}_{(n-1)\delta+\alpha_{1}}\mathcal{E}_{(m+1)\delta+\alpha_{1}}+(-1)^{n+m}\mathcal{E}_{(m+1)\delta+\alpha_{1}}\mathcal{E}_{(n-1)\delta+\alpha_{1}}.  
\end{align*}  
\end{lemma} 
\begin{proof}
We prove the lemma for indices \( n \) and \( m \) by double induction. First, we verify that the base case \( P(n,0) \) holds. Second, we assume that the lemma \( P(n,i) \) is true for all positive integers \( i \)  satisfying \( i<m \), then proceed to prove that \( P(n,m) \) holds.

(I) When \(m = 0\), let us consider the relations 
\begin{align}
\mathcal{E}_{n\delta+\alpha_1} e_1 + (-1)^n q^{-2} e_1 \mathcal{E}_{n\delta+\alpha_1} = q^{-2} \mathcal{E}_{(n-1)\delta+\alpha_1} \mathcal{E}_{\delta+\alpha_1} + (-1)^n \mathcal{E}_{\delta+\alpha_1} \mathcal{E}_{(n-1)\delta+\alpha_1}. 
\end{align}

We first examine the case of \(n=1\), then extend the result to general \(n \geq 1\).
 
Substituting \(n = 1\) and \(m = 0\) into (3.1), we need to prove 
\(
\mathcal{E}_{\delta+\alpha_1} e_1 =q^{-2} e_1 \mathcal{E}_{\delta+\alpha_1}.
\) It is now straightforward to verify that 
\[
\mathcal{E}_{\delta+\alpha_1} e_1 - q^{-2} e_1 \mathcal{E}_{\delta+\alpha_1}=\mathcal{T}_1\Phi\left(-e_1k^{-1}_2f_2 + q^{-2}k^{-1}_2f_2e_1\right).
\]    
We use \((A2)\) and \((A3)\) to get that
\(
\mathcal{E}_{\delta+\alpha_1} e_1 =q^{-2} e_1 \mathcal{E}_{\delta+\alpha_1}\).
 
Assume relation (3.1) holds for \(n = k\,(k>1)\). To confirm it holds for \(n = k+1\), we need to use the following identities
\begin{align}
[2]_q\mathcal{E}_{(t+1)\delta+\alpha_{1}} = \mathcal{E}_{\delta}\mathcal{E}_{t\delta+\alpha_{1}} + \mathcal{E}_{t\delta+\alpha_{1}}\mathcal{E}_{\delta},\quad \forall t\ge 0,\qquad (\text{by Theorem 2.10})
\end{align}
and use induction hypothesis to get
\begin{align}
\mathcal{E}_{s\delta+\alpha_1} e_1 + (-1)^s q^{-2} e_1 \mathcal{E}_{s\delta+\alpha_1} = q^{-2} \mathcal{E}_{(s-1)\delta+\alpha_1} \mathcal{E}_{\delta+\alpha_1} + (-1)^s \mathcal{E}_{\delta+\alpha_1} \mathcal{E}_{(s-1)\delta+\alpha_1},
\end{align}
where \(1\le s\le k\).

Applying (3.2) and (3.3), it is straightforward to verify that
\begin{align*}
[2]_q\mathcal{E}_{(k+1)\delta+\alpha_{1}}e_{1} &=\Bigl\{\mathcal{E}_{\delta}\mathcal{E}_{k\delta+\alpha_{1}} + \mathcal{E}_{k\delta+\alpha_{1}}\mathcal{E}_{\delta}\Bigr\}e_{1}\\
&=\mathcal{E}_{\delta}\Bigl\{(-1)^{k-1} q^{-2}e_1\mathcal{E}_{k\delta+\alpha_{1}}+q^{-2}\mathcal{E}_{(k-1)\delta+\alpha_{1}}\mathcal{E}_{\delta+\alpha_{1}}+(-1)^k\mathcal{E}_{\delta+\alpha_{1}}\mathcal{E}_{(k-1)\delta+\alpha_{1}}\Bigr\}\\
&\quad+\mathcal{E}_{k\delta+\alpha_{1}}\Bigl\{[2]_q\mathcal{E}_{\delta+\alpha_{1}} -e_1\mathcal{E}_{\delta}\Bigr\}\\
&=(-1)^{k-1} q^{-2}\Bigl\{[2]_q\mathcal{E}_{\delta+\alpha_{1}} -e_1\mathcal{E}_{\delta}\Bigr\}\mathcal{E}_{k\delta+\alpha_{1}}+q^{-2}\Bigl\{ [2]_q\mathcal{E}_{k\delta+\alpha_{1}}-\mathcal{E}_{(k-1)\delta+\alpha_{1}}\mathcal{E}_{\delta}\Bigr\}\mathcal{E}_{\delta+\alpha_{1}}
\\
&\quad+(-1)^k\Bigl\{[2]_q\mathcal{E}_{2\delta+\alpha_{1}}-\mathcal{E}_{\delta+\alpha_{1}}\mathcal{E}_{\delta}\Bigr\}\mathcal{E}_{(k-1)\delta+\alpha_{1}}+[2]_q\mathcal{E}_{k\delta+\alpha_{1}} \mathcal{E}_{\delta+\alpha_{1}}
\\
&\quad-\Bigl\{(-1)^{k-1} q^{-2}e_1\mathcal{E}_{k\delta+\alpha_{1}}+q^{-2}\mathcal{E}_{(k-1)\delta+\alpha_{1}}\mathcal{E}_{\delta+\alpha_{1}}+(-1)^k\mathcal{E}_{\delta+\alpha_{1}}\mathcal{E}_{(k-1)\delta+\alpha_{1}}\Bigr\}\mathcal{E}_{\delta}
\end{align*}
\begin{align*}
&=[2]_q\Bigl\{(-1)^{k}q^{-2}e_{1}\mathcal{E}_{(k+1)\delta+\alpha_{1}}+q^{-2}\mathcal{E}_{k\delta+\alpha_{1}}\mathcal{E}_{\delta+\alpha_{1}}+(-1)^{k+1}\mathcal{E}_{\delta+\alpha_{1}}\mathcal{E}_{k\delta+\alpha_{1}}\\
&\quad+\Bigl(\mathcal{E}_{k\delta+\alpha_{1}}\mathcal{E}_{\delta+\alpha_1}+(-1)^{k+1}q^{-2}\mathcal{E}_{\delta+\alpha_1}\mathcal{E}_{k\delta+\alpha_{1}}-q^{-2}\mathcal{E}_{(k-1)\delta+\alpha_{1}}\mathcal{E}_{2\delta+\alpha_{1}}\\
&\quad+(-1)^{k}\mathcal{E}_{2\delta+\alpha_{1}}\mathcal{E}_{(k-1)\delta+\alpha_{1}}\Bigr)\Bigr\}\\
&=[2]_q\Bigl\{(-1)^{k}q^{-2}e_{1}\mathcal{E}_{(k+1)\delta+\alpha_{1}}+q^{-2}\mathcal{E}_{k\delta+\alpha_{1}}\mathcal{E}_{\delta+\alpha_{1}}+(-1)^{k+1}\mathcal{E}_{\delta+\alpha_{1}}\mathcal{E}_{k\delta+\alpha_{1}}\\
&\quad+(\mathcal{T}_1\Phi)\Bigl(\mathcal{E}_{(k-1)\delta+\alpha_{1}}e_1+(-1)^{k+1}q^{-2}e_1\mathcal{E}_{(k-1)\delta+\alpha_{1}}-q^{-2}\mathcal{E}_{(k-2)\delta+\alpha_{1}}\mathcal{E}_{\delta+\alpha_{1}}\\ 
&\quad+(-1)^{k}\mathcal{E}_{\delta+\alpha_{1}}\mathcal{E}_{(k-2)\delta+\alpha_{1}}\Bigr)\Bigr\}.
\end{align*}
We ultimately obtain  
\begin{align*}
\mathcal{E}_{(k+1)\delta+\alpha_{1}}e_{1}+(-1)^{k+1}q^{-2}e_{1}\mathcal{E}_{(k+1)\delta+\alpha_{1}}
=q^{-2}\mathcal{E}_{k\delta+\alpha_{1}}\mathcal{E}_{\delta+\alpha_{1}}+(-1)^{k+1}\mathcal{E}_{\delta+\alpha_{1}}\mathcal{E}_{k\delta+\alpha_{1}}.
\end{align*}

Thus, relation (3.1) holds for all \(n > 0\).

(II) When \(m=t\), \(n>t>0\), we assume the lemma holds for \(m=t\):
\begin{align}
\mathcal{E}_{n\delta+\alpha_1} \mathcal{E}_{t\delta+\alpha_1} + (-1)^{n+t}& q^{-2} \mathcal{E}_{t\delta+\alpha_1} \mathcal{E}_{n\delta+\alpha_1} \nonumber\\
&= q^{-2} \mathcal{E}_{(n-1)\delta+\alpha_1} \mathcal{E}_{(t+1)\delta+\alpha_1} + (-1)^{n+t} \mathcal{E}_{(t+1)\delta+\alpha_1} \mathcal{E}_{(n-1)\delta+\alpha_1}.
\end{align}
We claim that the relation holds for \(m=t+1\).

To prove the claim, applying the automorphism \(\mathcal{T}_1\Phi\) to relation (3.4), we obtain
\begin{align}
\mathcal{E}_{(n+1)\delta+\alpha_1} \mathcal{E}_{(t+1)\delta+\alpha_1} + &(-1)^{n+t}q^{-2} \mathcal{E}_{(t+1)\delta+\alpha_1} \mathcal{E}_{(n+1)\delta+\alpha_1} \nonumber\\
&\quad = q^{-2} \mathcal{E}_{n\delta+\alpha_1}\mathcal{E}_{(t+2)\delta+\alpha_1} + (-1)^{n+t} \mathcal{E}_{(t+2)\delta+\alpha_1} \mathcal{E}_{n\delta+\alpha_1}.
\end{align}

We set \(n' = n+1\) and \(t' = t+1\). Since \(n > t\) implies \(n' > t'\), relation (3.5) simplifies to
\[
\begin{aligned}
\mathcal{E}_{n'\delta+\alpha_1} \mathcal{E}_{t'\delta+\alpha_1} + &(-1)^{n'+t'-1} q^{-2} \mathcal{E}_{t'\delta+\alpha_1} \mathcal{E}_{n'\delta+\alpha_1} \\
&\quad = q^{-2} \mathcal{E}_{(n'-1)\delta+\alpha_1} \mathcal{E}_{(t'+1)\delta+\alpha_1} + (-1)^{n'+t'-1} \mathcal{E}_{(t'+1)\delta+\alpha_1} \mathcal{E}_{(n'-1)\delta+\alpha_1},
\end{aligned}
\]
which matches the relation in the lemma for \(m = t' = t+1\), confirming the claim.

(III)
By the base case \(m = 0\), and inductive step, up to \(m = t+1\), the lemma holds.
\end{proof}

\begin{lemma} 
For \(n,m\in\mathbb{Z}_{\ge 0}\), \(n\le m\), the commutation relations between \(\mathcal{E}_{n\delta+\alpha_{2}}\) and \(\mathcal{E}_{m\delta+\alpha_{2}}\) are 
\begin{align*}  
\mathcal{E}_{n\delta+\alpha_{2}}&\mathcal{E}_{m\delta+\alpha_{2}}+(-1)^{n+m}q^{-2}\mathcal{E}_{m\delta+\alpha_{2}}\mathcal{E}_{n\delta+\alpha_{2}}\\ &=q^{-2}\mathcal{E}_{(n+1)\delta+\alpha_{2}}\mathcal{E}_{(m-1)\delta+\alpha_{2}}+(-1)^{n+m}\mathcal{E}_{(m-1)\delta+\alpha_{2}}\mathcal{E}_{(n+1)\delta+\alpha_{2}}. 
\end{align*}
\end{lemma}
\begin{proof} 
Similar to the proof of Lemma 3.7.
\end{proof}

\begin{lemma} 
$[\mathcal{E}_{m\delta},\mathcal{E}_{n\delta}]=0,$ for any $n,m\in\mathbb{Z}_+$.
\end{lemma}
\begin{proof}
See that of Proposition 2.1 in \cite{HZhuang}.
\end{proof}

\begin{lemma} 
For $n,m\in\mathbb{Z}_{\ge 0}$, $n>m$, the commutation relations between $\mathcal{E}_{n\delta+\alpha_i}$ and $\mathcal{F}_{m\delta+\alpha_j}\,(i,j\in\{1,2\})$ are 
\begin{align}  
&[\mathcal{E}_{m\delta+\alpha_1},\mathcal{F}_{m\delta+\alpha_1}]=\dfrac{k^{m+1}_1k^m_2- k^{-m-1}_1k^{-m}_2}{q - q^{-1}},\\ &[\mathcal{E}_{m\delta+\alpha_2},\mathcal{F}_{m\delta+\alpha_2}]=\dfrac{k^{m+1}_2k^m_1- k^{-m-1}_2k^{-m}_1}{q - q^{-1}},\\
&[\mathcal{E}_{n\delta+\alpha_1},\mathcal{F}_{m\delta+\alpha_1}] =-\mathcal{E}_{(n-m)\delta}k^{-m-1}_1k^{-m}_2,\\
&[\mathcal{E}_{n\delta+\alpha_2},\mathcal{F}_{m\delta+\alpha_2}] =k^{m}_1k^{m+1}_2\mathcal{E}_{(n-m)\delta}.
\end{align}  
\end{lemma} 
\begin{proof}
We proceed to verify equations (3.6)---(3.9).

(I) Verify equation (3.6). Applying \((\mathcal{T}_1\Phi)^n\) to  \([e_1, f_1]\) gives
\begin{align}
&[\mathcal{E}_{n\delta+\alpha_1}, \mathcal{F}_{n\delta+\alpha_1}] = (\mathcal{T}_1\Phi)^n\Bigl([e_1, f_1]\Bigr)=(\mathcal{T}_1\Phi)^n\Bigl(\dfrac{k_1 - k_1^{-1}}{q - q^{-1}}\Bigr).
\end{align}
Under the automorphism \((\mathcal{T}_1\Phi)^n\), the generator \(k_1\) transforms as \(k_1 \mapsto k_1^nk_2\). Thus,
\begin{align}
(\mathcal{T}_1\Phi)^n(k_1) = k_1^{n+1}k_2^n, \qquad (\mathcal{T}_1\Phi)^n(k_1^{-1}) = k_1^{-n-1}k_2^{-n}.
\end{align}
Substituting (3.11) into (3.10), we obtain
\begin{align*}
&[\mathcal{E}_{n\delta+\alpha_1}, \mathcal{F}_{n\delta+\alpha_1}] = \dfrac{k_1^{n+1}k_2^n - k_1^{-n-1}k_2^{-n}}{q - q^{-1}}.
\end{align*}
The proof of (3.7) is analogous to that of (3.6).

(II) Verify equation (3.8).
For \(n > m\), we first establish an intermediate claim
\begin{align}
[\mathcal{E}_{n\delta+\alpha_1}, f_1] = -\mathcal{E}_{n\delta}k_1^{-1},
\end{align}
where \(n\ge 1\).

When \(n = 1\), using \(\mathcal{E}_{\delta+\alpha_1} = (\mathcal{T}_1\Phi)(e_1)\) and \(\mathcal{T}_1(e_1) = -f_1k_1\), direct computation gives
\[
[\mathcal{E}_{\delta+\alpha_1}, f_1] = -[(\mathcal{T}_1\Phi)(e_1), \mathcal{T}_1(e_1k_1)] =  -\mathcal{T}_1[e_2, e_1k_1]=
-\mathcal{E}_{\delta}k_1^{-1},
\]
which confirms the claim to hold for \(n = 1\).

Assume the claim holds for \(n = k\). To prove that (3.12) holds for \(n = k+1\), we need the following equalities
\begin{align}
\mathcal{E}_{\delta}f_1 &= f_1\mathcal{E}_{\delta} - [2]_qe_2k_1^{-1},\\
[2]_q\mathcal{E}_{n\delta+\alpha_1} &= \mathcal{E}_{\delta}\mathcal{E}_{(n-1)\delta+\alpha_1} + \mathcal{E}_{(n-1)\delta+\alpha_1}\mathcal{E}_{\delta}.
\end{align}
By the induction hypothesis, we have 
\begin{align}
\mathcal{E}_{(n-1)\delta+\alpha_1}f_1 = f_1\mathcal{E}_{(n-1)\delta+\alpha_1} - \mathcal{E}_{(n-1)\delta}k_1^{-1}.
\end{align}

Substituting equations (3.13), (3.14) and (3.15) into the following calculation yields
\begin{align*}
[2]_q\mathcal{E}_{n\delta+\alpha_1}f_1&=\Bigl\{\mathcal{E}_{\delta}\mathcal{E}_{(n-1)\delta+\alpha_1} + \mathcal{E}_{(n-1)\delta+\alpha_1}\mathcal{E}_{\delta}\Bigr\}f_1\\
&=\mathcal{E}_{\delta}\Bigl(f_1\mathcal{E}_{(n-1)\delta+\alpha_1} - \mathcal{E}_{(n-1)\delta}k_1^{-1}\Bigr) + \mathcal{E}_{(n-1)\delta+\alpha_1}\Bigl(f_1\mathcal{E}_{\delta} - [2]_qe_2k_1^{-1}\Bigr)\\
&= \Bigl(f_1\mathcal{E}_{\delta} - [2]_qe_2k_1^{-1}\Bigr)\mathcal{E}_{(n-1)\delta+\alpha_1} - \mathcal{E}_{\delta}\mathcal{E}_{(n-1)\delta}k_1^{-1} \\
&\quad+ \Bigl(f_1\mathcal{E}_{(n-1)\delta+\alpha_1} - \mathcal{E}_{(n-1)\delta}k_1^{-1}\Bigr)\mathcal{E}_{\delta} - [2]_q\mathcal{E}_{(n-1)\delta+\alpha_1}e_2k_1^{-1} \\
&= [2]_qf_1\mathcal{E}_{(n+1)\delta+\alpha_1} - [2]_q\Bigl\{(-1)^n q^{-2}e_2\mathcal{E}_{(n-1)\delta+\alpha_1}+ \mathcal{E}_{(n-1)\delta+\alpha_1}e_2\Bigr\}k_1^{-1}.
\end{align*}

By the definition of \(\mathcal{E}_{n\delta}\), we obtain
\[
\mathcal{E}_{n\delta+\alpha_1}f_1 - f_1\mathcal{E}_{n\delta+\alpha_1} = -\mathcal{E}_{n\delta}k_1^{-1}.
\]
Thus, the claim holds for all \(n \ge 1\) by induction.

To extend to the case \(n > m\), we iterate the intermediate claim (3.12). We get
\[
[\mathcal{E}_{n\delta+\alpha_1}, \mathcal{F}_{m\delta+\alpha_1}] = \Bigl(\mathcal{T}_1\Phi\Bigr)^m\Bigl([\mathcal{E}_{(n-m)\delta+\alpha_1}, f_1]\Bigr)= -\mathcal{E}_{(n-m)\delta}k_1^{-m-1}k_2^{-m},
\]
which confirms equation (3.16). 

The proof for equation (3.9) follows similarly.
\end{proof}

\begin{lemma} 
For \(n\in\mathbb{Z}_{+}\), \(m\in\mathbb{Z}_{\ge 0}\), the commutation relations between \(\Psi(a(n))\) and \(\mathcal{E}_{m\delta+\alpha_i}\,(i\in\{1,2\})\) are  
\begin{align}
&[\Psi(a(n)), e_{1}]  = -\dfrac{[2n]_q}{n} \mathcal{E}_{n\delta+\alpha_1}(k_1k_2)^n,\\
&[\Psi(a(n)), e_2]  =(k_1k_2)^n\dfrac{[2n]_q}{n} \mathcal{E}_{n\delta+\alpha_2}.
\end{align}   
\end{lemma}  
\begin{proof}
Here we only prove equation (3.16), the proof of equation (3.17) is analogous and thus omitted. We complete the proof by  induction. First, we verify the base case \(n=1\), for which the conclusion obviously holds. 
Next, we assume  the equality is valid for \(n-1\), that is,
\begin{align}
[\Psi(a(l)), e_{1}]  = -\dfrac{[2l]_q}{l} \mathcal{E}_{l\delta+\alpha_1}(k_1k_2)^l,\quad  0<l\le n-1.
\end{align}
We then proceed to prove that the equality also holds for \(n\).

To this end, we need to recall some equalities given in \cite{11}.
\begin{align}
\mathcal{E}_{n\delta}e_1
=(-1)^{n}e_1\mathcal{E}_{n\delta}+\sum\limits_{k=1}^{n-1}b^{(n)}_k\mathcal{E}_{k\delta+\alpha_1}\mathcal{E}_{(n-k)\delta}+b^{(n)}_{n}\mathcal{E}_{n\delta+\alpha_1},
\end{align}
where \(b^{(n)}_k\) and \(b^{(n)}_n\) are coefficients
\begin{align}
&b^{(n)}_{k} = (-1)^{n-1}q^{-2(k-1)}(q^2 - q^{-2}), \\
&b^{(n)}_{n} = (-1)^{n-1}q^{-2(n-1)}[2]_q. 
\end{align}
Taking summation of both sides of equation (3.19), we directly obtain the following equality
\begin{align}
&\sum\limits_{r=1}^{n-1}[2r]_q\mathcal{E}_{(n-r)\delta}e_1\nonumber\\
=&\sum\limits_{r=1}^{n-1}[2r]_q\Bigl((-1)^{n-r}e_1\mathcal{E}_{(n-r)\delta}+\sum\limits_{k=1}^{n-r-1}b^{(n-r)}_k\mathcal{E}_{k\delta+\alpha_1}\mathcal{E}_{(n-r-k)\delta}+b^{(n-r)}_{n-r}\mathcal{E}_{(n-r)\delta+\alpha_1}\Bigr).
\end{align}
By applying the automorphism \((\mathcal{T}_1\Phi)^r\)  to equation (3.22), we obtain
\begin{align}
&\sum\limits_{r=1}^{n-1}[2r]_q\mathcal{E}_{(n-r)\delta}\mathcal{E}_{r\delta+\alpha_1}\nonumber\\
=&\sum\limits_{r=1}^{n-1}[2r]_q(\mathcal{T}_1\Phi)^r\Bigl((-1)^{n-r}e_1\mathcal{E}_{(n-r)\delta}{+}\sum\limits_{k=1}^{n-r-1}b^{(n-r)}_k\mathcal{E}_{k\delta+\alpha_1}\mathcal{E}_{(n-r-k)\delta}{+}b^{(n-r)}_{n-r}\mathcal{E}_{(n-r)\delta+\alpha_1}\Bigr)\nonumber\\
=&\sum\limits_{r=1}^{n-1}[2r]_q\Bigl((-1)^{n-r}\mathcal{E}_{r\delta+\alpha_1}\mathcal{E}_{(n-r)\delta}{+}\sum\limits_{k=1}^{n-r-1}b^{(n-r)}_k\mathcal{E}_{(k+r)\delta+\alpha_1}\mathcal{E}_{(n-r-k)\delta}{+}b^{(n-r)}_{n-r}\mathcal{E}_{n\delta+\alpha_1}\Bigr).
\end{align}
Recall the definition of \(\Psi(a(n))\):
\[
\Psi(a(n)) = (k_1k_2)^n\mathcal{E}_{n\delta} - \frac{q - q^{-1}}{n}\sum\limits_{r=1}^{n-1}r(k_1k_2)^{n-r}\mathcal{E}_{(n-r)\delta}\Psi(a(r)).
\]
Substituting equations (3.18), (3.19) and (3.23) into the following calculation yields
\begin{align*}
\Psi(a(n))e_1 &=\Bigl((k_1k_2)^n\mathcal{E}_{n\delta} - \frac{q - q^{-1}}{n}\sum\limits_{r=1}^{n-1}r(k_1k_2)^{n-r}\mathcal{E}_{(n-r)\delta}\Psi(a(r))\Bigr)e_1\\
&=(k_1k_2)^n \Bigl\{(-1)^ne_1\mathcal{E}_{n\delta} + \sum\limits_{r=1}^{n-1}b^{(n)}_r\mathcal{E}_{r\delta+\alpha_1}\mathcal{E}_{(n-r)\delta} + b^{(n)}_n\mathcal{E}_{n\delta+\alpha_1}\Bigr\}\\
&\quad- \frac{q - q^{-1}}{n}\sum\limits_{r=1}^{n-1}r(k_1k_2)^{n-r}\mathcal{E}_{(n-r)\delta}\Bigl(e_1\Psi(a(r))-\dfrac{[2r]_q}{r} \mathcal{E}_{r\delta+\alpha_1}(k_1k_2)^r\Bigr)
\\
&=(k_1k_2)^n \Bigl\{(-1)^ne_1\mathcal{E}_{n\delta} + \sum\limits_{r=1}^{n-1}b^{(n)}_r\mathcal{E}_{r\delta+\alpha_1}\mathcal{E}_{(n-r)\delta} + b^{(n)}_n\mathcal{E}_{n\delta+\alpha_1}\Bigr\}
\\
&\quad-\frac{q - q^{-1}}{n}\sum\limits_{r=1}^{n-1}r(k_1k_2)^{n-r}\Bigl\{(-1)^{n-r}e_1\mathcal{E}_{(n-r)\delta} {+} \sum\limits_{k=1}^{n-r-1}b^{(n-r)}_k\mathcal{E}_{k\delta+\alpha_1}\mathcal{E}_{(n-r-k)\delta}
\end{align*}
\begin{align*}
&\quad+ b^{(n-r)}_{n-r}\mathcal{E}_{(n-r)\delta+\alpha_1}\Bigr\}\Psi(a(r)) - \frac{q {-} q^{-1}}{n}\sum\limits_{r=1}^{n-1}(-1)^{r+1}[2r]_q(k_1k_2)^{n}\\
&\quad\cdot\Bigl\{(-1)^{n-r}\mathcal{E}_{r\delta+\alpha_1}\mathcal{E}_{(n-r)\delta}+\sum\limits_{k=1}^{n-r-1}b^{(n-r)}_k\mathcal{E}_{(k+r)\delta+\alpha_1}\mathcal{E}_{(n-r-k)\delta}{+} b^{(n-r)}_{n-r}\mathcal{E}_{n\delta+\alpha_1}\Bigr\}\\
&=e_1\Psi(a(n))+(k_1k_2)^n\Bigl(b^{(n)}_n+\dfrac{q-q^{-1}}{n}\sum\limits_{r=1}^{n-1}(-1)^r[2r]_qb^{(n-r)}_{n-r}\Bigr)\mathcal{E}_{n\delta+\alpha_1}\\
&\quad+\text{the remaining terms}.
\end{align*}
We proceed to prove that the remaining terms are precisely zero. We first fix the variable 
\(r\) and then consider the relevant terms that contain \(\mathcal{E}_{r\delta+\alpha_1}\) within the remaining terms:
\begin{align*}
&b^{(n)}_r(k_1k_2)^r(-1)^{n-s}\mathcal{E}_{r\delta+\alpha_1}(k_1k_2)^{n-r}\mathcal{E}_{(n-r)\delta}-\dfrac{q-q^{-1}}{n}(k_1k_2)^r\mathcal{E}_{r\delta+\alpha_1}\Bigl\{(n-r)b^{(r)}_r\Psi(a(n-r))\\
&+\sum\limits_{s=1}^{n-1-r}s(k_1k_2)^{n-r-s}(-1)^{n-r-s}b^{(n-s)}_r\mathcal{E}_{(n-r-s)\delta}\Psi(a(s))+(-1)^{r+1}[2r](k_1k_2)^{n-r}\mathcal{E}_{(n-r)\delta}\\
&+\sum\limits_{s=1}^{r-1}(-1)^{n-s+r+1}[2s](k_1k_2)^{n-r}b^{(n-s)}_{r-s}\mathcal{E}_{(n-r)\delta}\Bigr\}\\
=&\Bigl\{b^{(n)}_r+\dfrac{q{-}q^{-1}}{n}(-1)^{n}[2r]+\dfrac{q{-}q^{-1}}{n}\sum\limits_{s=1}^{r-1}(-1)^{s}[2s]b^{(n-s)}_{r-s}\Bigr\}(-1)^{n-r}(k_1k_2)^r\mathcal{E}_{r\delta+\alpha_1}(k_1k_2)^{n-r}\mathcal{E}_{(n-r)\delta}\\
&+\dfrac{q-q^{-1}}{n}(k_1k_2)^r\mathcal{E}_{r\delta+\alpha_1}\sum\limits_{s=1}^{n-r-1}s(-1)^{n-r-s+1}(k_1k_2)^{n-r-s}b^{(n-s)}_r\mathcal{E}_{(n-r-s)\delta}\Psi(a(s))\\
&-\dfrac{q-q^{-1}}{n}(n-r)b^{(r)}_r(k_1k_2)^r\mathcal{E}_{r\delta+\alpha_1}\Psi(a(n-r))
\\
=&\Bigl\{b^{(n)}_r+\dfrac{q{-}q^{-1}}{n}(-1)^{n}[2r]+\dfrac{q{-}q^{-1}}{n}\sum\limits_{s=1}^{r-1}(-1)^{s}[2s]b^{(n-s)}_{r-s}\Bigr\}(-1)^{n-r}(k_1k_2)^r\mathcal{E}_{r\delta+\alpha_1}(k_1k_2)^{n-r}\mathcal{E}_{(n-r)\delta}\\
&+\dfrac{q-q^{-1}}{n}(k_1k_2)^r\mathcal{E}_{r\delta+\alpha_1}
(-1)^rq^{-2(r-1)}(q^2-q^{-2})\dfrac{n-r}{q-q^{-1}}\Bigl\{(k_1k_2)^{n-r}\mathcal{E}_{(n-r)\delta}-\Psi(a(n-r))\Bigr\}
\\
&-\dfrac{q-q^{-1}}{n}(n-r)b^{(r)}_r(k_1k_2)^r\mathcal{E}_{r\delta+\alpha_1}\Psi(a(n-r))\\
=&0.
\end{align*}
To sum up, we have finished demonstrating that equation (3.16) holds.
\end{proof}

\begin{remark}
The commutation relations can be generalized to arbitrary real root vectors \(\mathcal{E}_{m\delta+\alpha_i}\) and negative root generators \(f_i\). For \(n \in \mathbb{Z}_{+}\) and \(m \in \mathbb{Z}_{\ge 0}\), the following hold
\begin{align}
&[\Psi(a(n)), \mathcal{E}_{m\delta+\alpha_1}] = -\dfrac{[2n]_q}{n} \mathcal{E}_{(m+n)\delta+\alpha_1}(k_1k_2)^n,\\
&[\Psi(a(n)), \mathcal{E}_{m\delta+\alpha_2}] = (k_1k_2)^n\dfrac{[2n]_q}{n} \mathcal{E}_{(m+n)\delta+\alpha_2},
\end{align}
\begin{align}
&\Psi(a(n))f_2 + (-1)^{n-1}f_2\Psi(a(n)) = -\dfrac{[2n]_q}{n}q^{-2}\mathcal{E}_{(n-1)\delta+\alpha_1}k^n_1k^{n+1}_2,
\\
&\Psi(a(n))f_1 + (-1)^{n-1}f_1\Psi(a(n)) = (-1)^{n}\dfrac{[2n]_q}{n}q^{-2}k^{n-1}_1k^{n}_2\mathcal{E}_{(n-1)\delta+\alpha_2}.
\end{align}
\end{remark}

To prove the commutation relations between \(\Psi(a(n))\) and \(\Psi(a(-m))\)\(\,(n,m\in\mathbb{Z}_+)\),  we have to find the commutation relations between the \({\mathcal{E}_{n\delta}}^{\prime}s\) and the \({\mathcal{F}_{m\delta}}^{\prime}s\). 
We omit the proofs of the following results.

\begin{coro} 
For \(n,m\in\mathbb{Z}_{+}\) with \(n<m\), it holds
 \begin{align}  
[\Psi(a(n)), \mathcal{F}_{m\delta}]=&-[2]_q\mathcal{F}_{(m-n)\delta}k^n_1k^n_2(k_1k_2-k^{-1}_1k^{-1}_2)\dfrac{[n]_{1,q}[n]_{K,1}}{n}\nonumber\\
=&-\mathcal{F}_{(m-n)\delta}k^n_1k^n_2\dfrac{[2n]_q}{n} \cdot (k^n_1k^n_2 - k^{-n}_1k^{-n}_2),
\end{align} 
an analogous formula holds by applying \(\Omega\) to (3.28). 
\end{coro}
\begin{proof}
To derive the above relation, we start with the definition of \(\mathcal{F}_{m\delta}\)  and substitute the relevant commutation relations (3.26) \& (3.27) step by step.

By the definition of negative imaginary root vectors, \(\mathcal{F}_{m\delta}\) can be expressed as
\begin{align}
\mathcal{F}_{m\delta} = \mathcal{F}_{(m-1)\delta+\alpha_2}f_1 + (-1)^m q^2 f_1 \mathcal{F}_{(m-1)\delta+\alpha_2}.
\end{align}
By applying the automorphism \((\mathcal{T}^{-1}_2\Phi)^{m-1}\)  to equation (3.26), we obtain
\begin{align}
\Psi(a(n))\mathcal{F}_{(m-1)\delta+\alpha_2}
=(-1)^n\mathcal{F}_{(m-1)\delta+\alpha_2}\Psi(a(n))-\dfrac{[2n]_q}{n}q^{-2}(\mathcal{T}^{-1}_2\Phi)^{m-n}(e_1)k^{n+m-1}_1k^{n+m}_2.
\end{align}
By directly substituting equations (3.27), (3.29), and (3.30) into the subsequent calculation, we obtain
\begin{align*}
\Psi(a(n))\mathcal{F}_{m\delta} &=\Psi(a(n))\Bigl\{ \mathcal{F}_{(m-1)\delta+\alpha_2}f_1 + (-1)^m q^2 f_1 \mathcal{F}_{(m-1)\delta+\alpha_2} \Bigr\}\\
&=\Bigl\{(-1)^n\mathcal{F}_{(m-1)\delta+\alpha_2}\Psi(a(n))-\dfrac{[2n]_q}{n}q^{-2}(\mathcal{T}^{-1}_2\Phi)^{m-n}(e_1)k^{n+m-1}_1k^{n+m}_2\Bigr\}f_1\\
&\,+(-1)^mq^2\Bigl\{(-1)^{n}f_1\Psi(a(n)) +(-1)^n\frac{[2n]_q}{n} q^{-2} k_1^{n-1} k_2^n \mathcal{E}_{(n-1)\delta+\alpha_2}
\Bigr\}\mathcal{F}_{(m-1)\delta+\alpha_2}\\
&=\mathcal{F}_{(m-1)\delta+\alpha_2}f_1\Psi(a(n))+(-1)^{m+n-1}\dfrac{[2n]_q}{n}\Bigl\{k^{n-1}_1k^{n}_2\mathcal{F}_{(m-1)\delta+\alpha_2}\mathcal{E}_{(n-1)\delta+\alpha_2}\\
&\,-(\mathcal{T}^{-1}_2\Phi)^{m-n}(e_1)f_1k^{n+m-1}_1k^{n+m}_2\Bigr\}+ (-1)^{m+n} \frac{[2n]_q}{n} k_1^{n-1} k_2^{n}\mathcal{E}_{(n-1)\delta+\alpha_2} \mathcal{F}_{(m-1)\delta+\alpha_2}\\
&\,+({-}1)^{m+n}q^2f_1 \Bigl\{ ({-}1)^n \mathcal{F}_{(m-1)\delta+\alpha_2}\Psi(a(n)) {-} \frac{[2n]_q}{n} q^{-2} (\mathcal{T}_2^{-1}\Phi)^{m-n}(e_1) k_1^{n+m-1} k_2^{n+m} \Bigr\}.
\end{align*}

After rearranging terms and canceling intermediate terms, we obtain
\begin{align*}
\Psi(a(n))\mathcal{F}_{m\delta} &= \mathcal{F}_{m\delta}\Psi(a(n)) + (-1)^{m+n} \frac{[2n]_q}{n} \Bigl\{ k_1^{n-1} k_2^n \bigl[ \mathcal{E}_{(n-1)\delta+\alpha_2}, \mathcal{F}_{(m-1)\delta+\alpha_2} \bigr] \\
&\quad+ \bigl[ (\mathcal{T}_2^{-1}\Phi)^{m-n}(e_1), f_1 \bigr] k_1^{n+m-1} k_2^{n+m} \Bigr\}.
\end{align*}
A simple inductive reasoning yields the following equation
\begin{align}
(\mathcal{T}_2^{-1}\Phi)^{m-n}(e_1) = -k_1^{n-m+1} k_2^{n-m} \mathcal{F}_{(m-n-1)\delta+\alpha_2}. 
\end{align}
Combining these results, the expression ultimately reduces to
\[
\Psi(a(n))\mathcal{F}_{m\delta} = \mathcal{F}_{m\delta}\Psi(a(n)) - \mathcal{F}_{(m-n)\delta}k_1^n k_2^n \cdot \frac{[2n]_q}{n} \cdot \left( k_1^n k_2^n - k_1^{-n} k_2^{-n} \right).
\]

This completes the proof.
\end{proof}

\begin{coro}
For \(n\in\mathbb{Z}_{+}\), it holds  
\begin{align}   
[\Psi(a(n)), \mathcal{F}_{n\delta}]
&=k_1^n k_2^n \dfrac{[2n]_q}{n} \cdot \dfrac{(k_1k_2)^{n} - (k_1k_2)^{-n}}{q - q^{-1}}.
\end{align} 
\end{coro}

\begin{proof}
Recall the definition of negative imaginary root vectors  
\begin{align}
\mathcal{F}_{n\delta}=\mathcal{F}_{(n-1)\delta+\alpha_2}f_1+(-1)^{n}q^2f_1\mathcal{F}_{(n-1)\delta+\alpha_2}.
\end{align}
By applying the automorphism \((\mathcal{T}^{-1}_2\Phi)^{n-1}\)  to equation (3.26), we obtain
\begin{align}
\Psi(a(n))\mathcal{F}_{(n-1)\delta+\alpha_2}
=(-1)^n\mathcal{F}_{(n-1)\delta+\alpha_2}\Psi(a(n))-\dfrac{[2n]_q}{n}q^{-2}e_1k^{2n-1}_1k^{2n}_2.
\end{align}
By directly substituting equations (3.27), (3.33), and (3.34) into the subsequent calculation, we obtain
\begin{align*}
\Psi(a(n))\mathcal{F}_{n\delta} &=\Psi(a(n))\Bigl\{ \mathcal{F}_{(n-1)\delta+\alpha_2}f_1 + (-1)^n q^2 f_1 \mathcal{F}_{(n-1)\delta+\alpha_2} \Bigr\}\\
&=\Bigl\{(-1)^n\mathcal{F}_{(n-1)\delta+\alpha_2}\Psi(a(n))-\dfrac{[2n]_q}{n}q^{-2}e_1k^{2n-1}_1k^{2n}_2\Bigr\}f_1\\
&\quad+(-1)^nq^2\Bigl\{(-1)^{n-1}f_1\Psi(a(n)) - \frac{[2n]_q}{n} q^{-2} k_1^{n-1} k_2^n \mathcal{E}_{(n-1)\delta+\alpha_2}
\Bigr\}\mathcal{F}_{(n-1)\delta+\alpha_2}\\
&=\mathcal{F}_{(n-1)\delta+\alpha_2}f_1\Psi(a(n))-\dfrac{[2n]_q}{n}\Bigl\{k^{n-1}_1k^{n}_2\mathcal{F}_{(n-1)\delta+\alpha_2}\mathcal{E}_{(n-1)\delta+\alpha_2}\\
&\quad-e_1f_1k^{2n-1}_1k^{2n}_2\Bigr\}+ \frac{[2n]_q}{n} k_1^{n-1} k_2^{n}\mathcal{E}_{(n-1)\delta+\alpha_2} \mathcal{F}_{(n-1)\delta+\alpha_2}\\
&\quad+q^2f_1 \Bigl\{ (-1)^n \mathcal{F}_{(n-1)\delta+\alpha_2}\Psi(a(n)) - \frac{[2n]_q}{n} q^{-2} e_1 k_1^{2n-1} k_2^{2n} \Bigr\}.
\end{align*}

After rearranging terms and cancelling intermediate terms, we obtain
\begin{align}
\Psi(a(n))\mathcal{F}_{n\delta} &= \mathcal{F}_{n\delta}\Psi(a(n)) + \frac{[2n]_q}{n} \Bigl\{ k_1^{n-1} k_2^n \bigl[ \mathcal{E}_{(n-1)\delta+\alpha_2}, \mathcal{F}_{(n-1)\delta+\alpha_2} \bigr] \nonumber\\
&\quad+ \bigl[ e_1, f_1 \bigr] k_1^{2n-1} k_2^{2n} \Bigr\}.
\end{align}
Combining these results, the expression ultimately reduces to
\[
\Psi(a(n))\mathcal{F}_{n\delta} = \mathcal{F}_{n\delta}\Psi(a(n)) +k_1^n k_2^n \cdot \frac{[2n]_q}{n} \cdot \left( k_1^n k_2^n - k_1^{-n} k_2^{-n} \right).
\]

This completes the proof.
\end{proof}

\begin{coro} 
For \(n,m\in\mathbb{Z}_{+}\) with \(n>m\), it holds \begin{align}   
[\Psi(a(n)), \mathcal{F}_{m\delta}]=0, 
\end{align} 
an analogous formula holds by applying \(\Omega\) to (3.36).
\end{coro} 
\begin{proof}  
Similar to that of Corollary 3.14.
\end{proof} 

\begin{lemma}
For \(n,m\in\mathbb{Z}_{+}\), the commutation relations between \(\Psi(a(n))\) and \(\Psi(a(-m))\) satisfy  
\begin{align} 
[\Psi(a(n)), \Psi(a(-m))]  = \delta_{n-m,0}\dfrac{[2n]_q}{n} \cdot \dfrac{(k_1k_2)^{n} - (k_1k_2)^{-n}}{q - q^{-1}}.
\end{align}    
\end{lemma}  

\begin{proof}
We analyze two cases based on the relationship between \(n\) and \(m\).  

(I) When \(n > m\).  By the definition of \(\Psi(a(-m))\) in Theorem 3.5, we have  
\begin{align}
\Psi(a(-m)) = \mathcal{F}_{m\delta}(k_1k_2)^{-m} + \frac{q-q^{-1}}{m}\sum_{r=1}^{m-1}r\Psi(a(-r))\mathcal{F}_{(m-r)\delta}(k_1k_2)^{-m+r}.
\end{align}

By induction on \(r\),  
\(\Psi(a(n))\) commutes with \(\Psi(a(-r))\). Since \(n > m\), \(\Psi(a(n))\) commutes with all terms on the right-hand side of equation (3.38). 
Thus
\begin{align}
[\Psi(a(n)),\Psi(a(-m))]=0,\quad n>m.
\end{align}

(II) When \(n = m\).
Substituting equations (3.32), (3.36), and (3.39) into the following equation and performing the calculation, we obtain
\begin{align*}
\Psi(a(n))\Psi(a(-n))
&= \Psi(a(n))\Bigl(\mathcal{F}_{n\delta}(k_1k_2)^{-n} + \frac{q-q^{-1}}{n}\sum_{r=1}^{n-1}r\Psi(a(-r))\mathcal{F}_{(n-r)\delta}(k_1k_2)^{-n+r}\Bigr)
\\
&=\mathcal{F}_{n\delta}(k_1k_2)^{-n}\Psi(a(n)) +  \dfrac{[2n]_q}{n} \cdot \dfrac{(k_1k_2)^n - (k_1k_2)^{-n}}{q - q^{-1}}
\\
&\quad+\dfrac{q-q^{-1}}{n}\sum_{r=1}^{n-1}r\Psi(a(-r))\mathcal{F}_{(n-r)\delta}(k_1k_2)^{-n+r}\Psi(a(n))
\end{align*}
\begin{align*}
&=\Psi(a(-n))\Psi(a(n))+\dfrac{[2n]_q}{n} \cdot \dfrac{(k_1k_2)^n - (k_1k_2)^{-n}}{q - q^{-1}}.
\end{align*} 
 
From the conclusions of (I) and (II), we prove the validity of this lemma.
\end{proof}
\begin{lemma}
\(\Psi\) is an algebraic homomorphism.
\end{lemma}
\begin{proof}
We confirm that \(\Psi\) preserves the defining relations (D1)---(D6) of \(U_{\textbf{q}}^{D}(\widehat{\mathfrak{sl}}_{2})\). It is clear that  \(\Psi\) preserves (D1) and (D2), we are left to verify (D3)---(D6).

For (D3), it follows from lemma 3.16.

For (D4), since \(\Psi(x^{+}(k))=(\mathcal{T}_1\Phi)^k(e_1)\)
and \(\Psi(x^{-}(k))=(\mathcal{T}_1\Phi)^{-k}(f_1)\), applying \((\mathcal{T}_1\Phi)^k\)
and 
\((\mathcal{T}_1\Phi)^{-k}\)
to both sides of equations (3.16) and (3.27) respectively, it follows that \(\Psi\) preserves  (D4).

For (D5), from lemmas 3.7 and 3.8, it follows that \(\Psi\) preserves (D5).

(D6) is preserved through  the commutator relations such as \([\mathcal{E}_{n\delta+\alpha_i}, \mathcal{F}_{m\delta+\alpha_j}]\).

We have thus completed the proof of Theorem 3.5.
\end{proof}
\subsection{Proof of Theorem 3.6}

\textbf{Theorem 3.6.}
\textit{There exists a surjective} \(\Xi: U_{\textbf{q}}(\widehat{\mathfrak{sl}}_{2}) \to U_{\textbf{q}}^{D}(\widehat{\mathfrak{sl}}_{2})\) such that \(\Psi\Xi=\Xi\Psi=\text{\rm id}\).
\begin{proof}
We define \(\Xi\) on the generators as follows
\begin{align*}
&k_{1} \longmapsto \gamma k^{-1}_2, \quad e_{1} \longmapsto x^{+}(0), \quad f_{2} \longmapsto -k_2x^{+}(-1), \\
&k_{2} \longmapsto k_2, \qquad f_{1} \longmapsto x^{-}(0), \quad e_{2} \longmapsto -x^{-}(1)k^{-1}_2.
\end{align*}
Consequently, it is not difficult to see that \(\Xi \Psi = \Psi\Xi=\text{\rm id}\).  
\end{proof}
Up to now, we have proved the Drinfeld Isomorphism Theorem for the novel quantum affine algebra \(U_{\textbf{q}}(\widehat{\mathfrak{sl}}_2)\).

\section*{Acknowledgements}
R.S. Zhuang is supported by the Scientific Research Funds at China University of
Geosciences (Wuhan) (Project No. 2025017). G. Feng is supported by the NSFC Grant No. 12526586. N.H. Hu is supportedby the NSFC Grant No. 12171155, and in part supported by the Science and Technology Com-mission of Shanghai Mumicipality (Project No. 22DZ2229014).




\end{document}